\documentclass{amsart}
\usepackage{a4wide}
\usepackage{amsfonts,amsmath,amsthm,amssymb}
\usepackage[linesnumbered,ruled,noend]{algorithm2e}
\usepackage{comment}
\usepackage{float}
\usepackage{latexsym}
\usepackage{graphicx}
\usepackage{multicol}
\usepackage{color}
\usepackage{xcolor}
\usepackage{pifont}
\usepackage{enumerate}
\usepackage{lmodern}
\usepackage{makecell}
\usepackage{adjustbox}
\usepackage{hyperref}
\usepackage{setspace}
\usepackage[all]{xy}
\usepackage{pb-diagram,pb-xy}
\usepackage{pdflscape}
\usepackage{chngpage}
\usepackage{breakurl}
\usepackage{adjustbox,lipsum}
\usepackage{graphicx}
\usepackage[justification=centering]{caption}
\usepackage{float}

\usepackage{tabu}

\makeatletter
\def\BState{\State\hskip-\ALG@thistlm}
\makeatother

\theoremstyle{plain}
\newtheorem{theorem}{Theorem}[section]
\newtheorem{lemma}[theorem]{Lemma}
\newtheorem{proposition}[theorem]{Proposition}
\newtheorem{corollary}[theorem]{Corollary}

\theoremstyle{definition}

\newtheorem*{example*}{Example}

\newcommand{\F}{\mathbb{F}}
\newcommand{\Q}{\mathbb{Q}}
\newcommand{\Z}{\mathbb{Z}}

\newcommand{\q}{\mathfrak{q}}

\DeclareMathOperator{\ord}{ord}

\begin{document}

\title[]{On perfect powers that are sums of cubes of a nine term arithmetic progression}

\author{Nirvana Coppola}
\address{Vrije Universiteit Amsterdam, de Boelelaan 1111, Room 9A94, 1081 HV, Amsterdam, The Netherlands}
\email{nirvanac93@gmail.com}

\author{Mar Curc\'o-Iranzo}
\address{Hans-Freudental Gebouw, Utrecht University, Budapestlaan 6, Room 5.03, 3584 CD Utrecht, The Netherlands}
\email{m.curcoiranzo@uu.nl}

\author{Maleeha Khawaja}
\address{School of Mathematics and Statistics, University of Sheffield, Hounsfield Road, Sheffield S3 7RH, United Kingdom}
\email{mkhawaja2@sheffield.ac.uk}

\author{Vandita Patel}
\address{School of Mathematics, University of Manchester, Oxford Road, Manchester M13 9PL, United Kingdom}
\email{vandita.patel@manchester.ac.uk}

\author{\"{O}zge \"{U}lkem}
\address{Galatasaray University, \c{C}ıra\u{g}an Cd. No:36, Istanbul, Turkey}
\email{ozgeulkem@gmail.com}

\date{\today}

\keywords{Exponential equation, Baker's Bounds, Thue equation, 
Lehmer sequences, primitive divisors. \\
Data sharing not applicable to this article as no datasets were generated or analysed during the current study.}

\begin{abstract}

We study the equation 
$(x-4r)^3 + (x-3r)^3 + (x-2r)^3+(x-r)^3 + x^3 + (x+r)^3+(x+2r)^3 + (x+3r)^3 + (x+4r)^3 = y^p$, which is a natural continuation of previous works carried out by A. Arg\'{a}ez-Garc\'{i}a and the fourth author (perfect powers that are sums of cubes of a three, five and seven term arithmetic progression). Under the assumptions 
 $0 < r \leq 10^6$, $p \geq 5 $ a prime  and  $\gcd(x, r) = 1$, we show that solutions must 
satisfy $xy=0$. Moreover, we study the equation for prime exponents $2$ and $3$ in greater detail. Under the assumptions 
 $r>0$  a positive integer and  $\gcd(x, r) = 1$  we show that there are infinitely many solutions for $p=2$ and $p=3$ via explicit constructions using integral points on elliptic curves.  We use an amalgamation of methods in computational and algebraic number theory to overcome the increased computational challenge. Most notable is a significant computational efficiency obtained through appealing to Bilu, Hanrot and Voutier's Primitive Divisor Theorem and the method of Chabauty, as well as employing a Thue equation solver earlier on.

\end{abstract}
\maketitle

\section{Introduction}

 Solving Diophantine equations has always been the order of the day amongst number theorists. 
 In this paper, we explore Diophantine equations that arise from arithmetic progressions. These types of equations have been previously studied by many; 
 see \cite{bartoli2019diophantine}, \cite{bennett2020equation}, \cite{Bennett2016}, \cite{bennett2017perfect}, \cite{berczes2018diophantine}, \cite{BreStrTza}, \cite{Cassels}, \cite{cohn_1996},  
 \cite{koutsianas2018perfect}, \cite{kundu2021perfect}, 
 \cite{Patel17}, \cite{pinter2007power}, \cite{Stroeker1995}, \cite{soydan2017diophantine}, \cite{Uchiyama}, \cite{van2021sum}, \cite{zhang2014diophantine}, \cite{zhang2017diophantine4}, \cite{zhang2017diophantine} and \cite{bai2013diophantine}. We refer the reader to the state-of-the-art survey \cite[Section 2]{CCIKPUSurvey} for a comprehensive overview of these works.
 
 \medskip
 
 Specifically, we determine perfect powers that can be written as the sum of cubes of nine consecutive terms in a bounded arithmetic progression when the exponent $p \geq 5$ is a prime. 
This is a natural continuation of previous work carried out by A. Arg\'{a}ez-Garc\'ia and the fourth author (see \cite{Garcia2019}, \cite{GarciaPatel2019} and \cite{GarciaPatel2020}).
In particular, we prove the following theorem:
\begin{theorem} \label{thm:main}
Let $p \geq 5$ be a prime and $0 < r \leq 10^6$. The equation
\begin{equation} 
\label{eq:mainthm} 
(x-4r)^3 + (x-3r)^3 + (x-2r)^3+(x-r)^3 + x^3 + (x+r)^3+(x+2r)^3 + (x+3r)^3 + (x+4r)^3 = y^p \end{equation}
with $x, y, p \in \Z$ and $\gcd(x, r) = 1$ only has integer solutions which satisfy $xy = 0$.
\end{theorem}

Our proof is inspired by  \cite{bennett2017perfect}, \cite{Patel17} and \cite{VP21}.  The proof of Theorem~\ref{thm:main} 
uses a battery of different 
techniques used in the realm of  Diophantine equations. Our careful fusion of various 
techniques 
allows us to overcome the increased computational challenge occurred when allowing for more terms in  arithmetic progression, thus enabling the proof of 
Theorem~\ref{thm:main}.

\medskip

Indeed, the novelties of this paper include a significant computational efficiency obtained through appealing to
Bilu, Hanrot and Voutier’s Primitive Divisor Theorem and the method of Chabauty, as well as
employing a Thue equation solver earlier on. Section~\ref{sec:compdat} reports the significant computational savings obtained.

\medskip

Before we dive into the details, we give an overview of the structure of the proof. In Section \ref{sec:prime23}, we consider the small prime exponents $2$ and $3$, and prove the infinitude of solutions. 
Subsequently, we may assume that the exponent is greater than or equal to $5$.  In Section \ref{sec:descandbound}, we apply a descent argument to equation~\eqref{eq:mainthm}, resulting in $12$ distinct ternary equations that require resolving.
Thus, in order to prove Theorem ~\eqref{thm:main}, we need to show that any solution arising from any descent case corresponds to a trivial solution to ~\eqref{eq:mainthm} (i.e. $xy=0$).

\medskip

We then apply a theorem of Mignotte \cite{Mignotte}, \cite[Chapter 12, page 423]{Cohen} which is based on the theory of linear forms in logarithms to bound the exponent $p$.
This is an integral step since it reduces the proof of Theorem~\ref{thm:main} to the resolution of a finite number of equations in one less variable. 
However, we  end up obtaining roughly twenty billion equations to solve, in the unknown variables $x$ and $y$, which makes the implementation non-effective. 

\medskip

In Section \ref{sec:bhvbound}, we significantly reduce the number of equations that need to be resolved using the aforementioned work of the fourth author \cite{VP21} (which builds upon the Primitive Divisor Theorem \cite{BHV}, see Table \ref{tab:countingeqs}). 
To ease the computational burden, in Section \ref{sec:prime57} we 
treat the prime exponents $5$ and $7$ separately where possible. We employ the method of Chabauty (see \cite{mccallum2012method}, \cite{Siksek2015Chabauty} and \cite{stoll2001implementing}), in combination with the computationally efficient test presented in \cite{VP21}, as well as making use of \texttt{MAGMA's} \cite{magma} inbuilt Thue solver, which is based on an algorithm of Bilu and Hanrot \cite{BILU1996373}, and Tzanakis and de Weger \cite{TW89}.

\medskip

To further our quest for elimination,  we then apply the ``empty set" criteria which is based on work of Sophie Germain
(see \cite[Section 1.3]{CCIKPUSurvey} or \cite[pp. 358-359]{centina} for the original statement). 
We try to eliminate the remaining equations by performing a further descent over number fields and applying local solubility tests, all of which are detailed in Section~\ref{sec: milliontosolve}.
The implementation of these tests in \texttt{MAGMA} \cite{magma} will point us towards 
 descent equations with specified prime exponent $p$  
 with potential non-trivial solutions. We solve the last remaining equations with \texttt{MAGMA}'s 
 inbuilt \texttt{Thue} solver. 
 We refer the keen and interested reader to \cite{CCIKPUSurvey} for a comprehensive overview of existing results in this area as well as an overview of contemporary and classical techniques used in the resolution of Diophantine equations.

\medskip

In contrast to previous works, we also carry out studies for the exponent $p=2$ and $p=3$. We prove the following theorem and are able to provide an explicit construction for solutions.
\begin{theorem} \label{thm:main2}
Let $p =2$ or $p=3$.  The equation
\begin{equation*} 
(x-4r)^3 + (x-3r)^3 + (x-2r)^3+(x-r)^3 + x^3 + (x+r)^3+(x+2r)^3 + (x+3r)^3 + (x+4r)^3 = y^p \end{equation*}
with $x, r, y \in \Z$, $r>0$ and $\gcd(x, r) = 1$ has infinitely many non-trivial solutions.
\end{theorem} 
The proof of Theorem~\ref{thm:main2} can be found in Section~\ref{sec:prime23}.

\section*{Acknowledgements}

This project stemmed from the Women in Numbers Europe 4 workshop, which took place in August 2022 at Utrecht University. The authors are  immensely appreciative towards the organisers: Ramla Abdellatif,
Valentijn Karemaker,
Ariane M\'ezard and Nirvana Coppola for hosting such an inspiring and productive workshop, and for all of their time committed towards such a noble  
endeavour. 

N. Coppola is supported by the NWO Vidi grant No. 639.032.613, New Diophantine Directions.
M. Khawaja is supported by an EPSRC studentship from the University of Sheffield (EPSRC grant no. EP/T517835/1).
\"{O}. \"{U}lkem is supported by T\"{U}BITAK project no. 119F405.

\section{Small prime exponents}
\label{sec:prime23} 

In this section, we treat equation~\eqref{eq:mainthm} for $p \in \{2,3\}$, thus proving the following result.

\begin{theorem}
    Let $p \in \{2,3\}$. 
    The equation
\begin{equation*} 
(x-4r)^3 + (x-3r)^3 + (x-2r)^3+(x-r)^3 + x^3 + (x+r)^3+(x+2r)^3 + (x+3r)^3 + (x+4r)^3 = y^p 
\end{equation*}
with $x, r, y \in \Z$, $\gcd(x, r) = 1$ and $r >0$ has infinitely many integer solutions.
\end{theorem}

\subsection*{The exponent $2$.} 
If $p=2$ in equation~\eqref{eq:mainthm}, we obtain an elliptic curve, namely $9x(x^2~+~20r^2) = y^2$. The change of variables $y = 3Y$ and  $x=X$ gives rise to the  following integral Weierstrass model:
\begin{equation} \label{eq:ellcvp2}
    E_r : Y^2 = X^3 + 20 r^2 X.
\end{equation} 

The curious reader can find a detailed overview of the elliptic curves theory used in this section in \cite{silverman2009arithmetic}.

\medskip

Integral points on $E_r$ correspond to integral solutions to equation~\eqref{eq:mainthm} when $p=2$.
For a fixed value of $r$, Siegel's Theorem \cite{Siegel} tells us that $E_r$ has finitely many integer points and with the help of computer software, we may be able to determine all integral points on $E_r$ for specified positive integer $r$. As we vary $r$, we obtain infinitely many integral points, hence infinitely many integral solutions to equation~\eqref{eq:mainthm}.

\medskip 
We demonstrate this  explicitly by constructing parametric families of solutions to equation~\eqref{eq:mainthm} when $p=2$.
We first ascertain that integral solutions to
equation~\eqref{eq:mainthm} do not arise from torsion points on $E_r$. 

\begin{theorem} \label{thm:torsion}
Let $r$ be a non-zero integer. 
    The curve $ E_r: Y^2 = X^3 + 20 r^2 X$ 
    has torsion subgroup $E_r(\Q)_{\text{tor}} = \{(0,0), \infty\}$.
\end{theorem}

\begin{proof}

The polynomial $X^3 + 20 r^2 X$ has two irreducible factors, namely $X$ and $X^2+20r^2$. Since $X^2+20r^2$ has no rational roots for any positive integer $r$, 
$\Z/2\Z \times \Z/2\Z \not\subseteq E(\Q)_{\text{tor}}$ but $\Z/2\Z \subseteq E(\Q)_{\text{tor}}$. Mazur's classification theorem \cite{Mazur77, Mazur78} gives the following possible options for $E(\Q)_{\text{tor}}$:
\[
\Z/2\Z, \quad \Z/4\Z, \quad \Z/6\Z, \quad \Z/8\Z, \quad  \Z/10\Z \quad \text{and} \quad \Z/12\Z.
\]

\medskip

To narrow down the possibilities further, we first  investigate the feasibility of 3-torsion. The 3-torsion points of $E_r$ are given by solutions of the $3$-division polynomial, 
\[
\psi_{3}(X) = 3X^{4}+6\cdot 20 r^{2}X^{2}-20^{2}r^{4};
\]
hence, the $x$-coordinates of the $3$-torsion points satisfy
\[
x^{2} = \frac{-6\cdot 20 r^{2} \pm 4 \cdot 20 r^{2} \sqrt{3}}{6}.
\]
Therefore, $x$ cannot be rational and so we have no $3$-torsion. 
Our list has now reduced to 
\[
\Z/2\Z, \quad \Z/4\Z, \quad \Z/8\Z \quad  \text{and} \quad \Z/10\Z.
\]

\medskip 

To refine the torsion group further, we must analyse 
the  $4$-th and $5$-th division polynomials. Using \texttt{MAGMA} \cite{magma}, and arguments similar to above, we find that there is  no rational $4$ or $5$-torsion. 
%
This leaves 
$$E(\Q)_{\text{tor}} = \Z/2\Z  $$
as our only possibility. In particular, $E(\Q)_{\text{tor}} = \{(0,0), \infty\}$.

\end{proof}

In particular, as a Corollary to Theorem~\ref{thm:torsion}, we note that the torsion points found on $E_r$ do not give rise to a non-trivial integral solution to ~\eqref{eq:mainthm} when $p=2$. However, we can still give an explicit construction of an infinite family of integral solutions to equation~\eqref{eq:mainthm} when $p=2$  using the curves $E_r$.

\begin{proposition}\label{prop:positiverank} 
Let $A$ and $B$ be non-zero integers
and take $r=2AB(A^{2}+5B^{2})$. 
Then, the elliptic curve $E_{r}$ as defined in \ref{eq:ellcvp2} has positive rank.

\end{proposition}

\begin{proof}

Let $A, B 
\in \Z \setminus \{0\}$. 
 Define 
 \[
X=(A^{2}-5B^{2})^2, \qquad
Y= (A^{2}-5B^{2}) ( (A^{2}+5B^{2})^2 + 20A^2B^2),
\]
which gives an integral point $(X,Y)$ on $E_r$.
Thus, given non-zero integers $A$ and $B$, and setting  $r=2AB(A^{2}+5B^{2})\in \Z\setminus \{0\}$, 
we have explicitly 
constructed 
 an integer point $(X,Y)$ different from $(0,0)$ on  $E_{r}$. 
Given Theorem~\ref{thm:torsion}, this must be a point of infinite order on $E_{r}$, and thus $E_{r}$ has strictly positive rank. 
\end{proof}


The following
corollary addresses our question of integral solutions to equation~\eqref{eq:mainthm} for 
$p=2$.

\begin{corollary}\label{cor:exp2}
Let $A, B$ be non-zero positive integers. 
Take $r=2AB(A^{2}+5B^{2})$.
Then equation~\eqref{eq:mainthm} has a non-trivial integral solution $(x,y)$ when $p=2$.
\end{corollary}

\begin{proof}
Proposition~\ref{prop:positiverank} gives an explicit (non-torsion hence non-trivial) integer point on the  curve $E_r$. 
If we take  
 \[
x=A^{2}-5B^{2}, \qquad
y= 3(A^{2}-5B^{2}) ( (A^{2}+5B^{2})^2 + 20A^2B^2),
\]
we obtain a non-trivial integral solution to equation~\eqref{eq:mainthm}.
\end{proof}

For completeness, we 
 note that we can pull up an infinite parametric family of solutions to equation~\eqref{eq:mainthm} when $p=2$ via a more elementary approach. For $M$ and $N$ positive integers, we let 
 \begin{equation*}
    x = 20(MN)^2,\qquad r = M^4-5N^4,\qquad y = 60MN(M^4+5N^4),
\end{equation*}
to obtain non-trivial integral solutions to
equation~\eqref{eq:mainthm} for $p=2$.

\medskip
Moreover, note that this parametric family of solutions then gives rise to integer points on $E_r$ for $r = M^4 - 5N^4$ where $M$ and $N$ are non-zero integers. Again, via Theorem~\ref{thm:torsion}, we may deduce that any member of the family of elliptic curves $E_r$ for $r=M^4 - 5N^4$ have positive rank.

\begin{corollary}\label{cor:exp2pt2}
Let $M$ and $N$ be strictly positive integers. Let $r= M^4 - 5N^4$. Take  
 \begin{equation*}
    X = 20(MN)^2,\qquad  Y = 20MN(M^4+5N^4).
\end{equation*}
Then $(X,Y)$ is an integral point on $E_r$.
\end{corollary}

\medskip

We immediately notice that the two parametric families of solutions  constructed have zero intersection. Suppose $P=(x_{P}, y_{P})\in E_{r}(\Q)$. For the family of points constructed in Corollary~\ref{cor:exp2}, $x_{P}$ is a square in $\Q$. Conversely, $x_{P}$ is never a square in $\Q$ for the family of points constructed in Corollary~\ref{cor:exp2pt2}.

\subsection*{The exponent $3$.} 
We now turn to the exponent $p=3$ case, which also yields a genus $1$ curve,  namely, 
\[
C: 9x(x^2 + 20r^2)=y^3.
\]

As this is a plane cubic with one rational point, it is an elliptic curve. More precisely, let 
$P=(x,y,r)=(0,0,1) \in C(\mathbb{Q})$. Then, using \texttt{Magma} \cite{magma}, we find that $C$ is isomorphic to an elliptic curve $E$ via the isomorphism
\begin{align*}
    \phi & : C \rightarrow E,\\
    \phi&(x,y,r)=(5y/x,-150r/x),
\end{align*}
where $E$ is the elliptic curve with Weierstrass equation
\[
E: Y^2 = X^3 - 1125.
\]
Observe that the curve $E$ is independent of $r$, and $E(\Q) \cong \Z$ with Mordell--Weil basis given by $(45, 300)$, which corresponds to $(x,y,r)=(1,9,-2)$. We deduce that  $C$ has infinitely many \emph{integer} points $(x,y,r)$.

\section{The first descent and initial exponent bound} 
 \label{sec:descandbound}

Since we have considered prime exponents 2 and 3 in Section~\ref{sec:prime23}, we now proceed under the assumption that $p\geq 5$. 

\medskip 

We rewrite equation~\eqref{eq:mainthm} as $9x(x^2+20r^2)= y^p$. Since $3 \mid y$, we make the substitution $y = 3w$ to obtain
\[
x(x^2+20r^2) = 3^{p-2}w^p.
\]
We note that $\gcd(x, x^2 + 20r^2) \in \{1,2,4,5, 10, 20\}$ depending on whether $2,4,5,10$ or $20$ divides $x$ or not. 
Therefore, we consider twelve cases and apply a simple descent argument in each case. 

\medskip

For eight of the cases, we bound the exponent $p$ by applying the following theorem of Mignotte \cite{Mignotte},\cite[Chapter 12, page 423]{Cohen} (based on the method of linear forms in logarithms).
The bounds obtained are recorded in Table~\ref{tab:descent}, along with the descent information.

\begin{theorem}[Mignotte]
    \label{thm:Mignotte}
Assume that the exponential Diophantine inequality 
\[
|ax^n - by^n | \leq c, \quad \text{ with } a,b,c \in \Z_{\ge 0} \text{ and }  a\neq b,
\]
has a solution in strictly positive integers $x$ and $y$ with $\max\{x,y\} > 1$. Let $A = \max \{a,b,3\}$. Then
\[
n \leq \max \left\{ 3 \log\left(1.5| c/b| \right),\; \frac{7400\log A}{\log\left(1+\log A/\mid \log a/b\mid\right)}\right\}.
\]
\end{theorem}

\begin{table}[htbp!]
\begin{adjustbox}{width =\textwidth}
\begin{tabular}{|c|c|c|c|c|}
\hline
{\bf Case} & {\bf Conditions on $x$} & {\bf Descent equations} & {\bf Ternary equation} & {\bf  Upper Bound for $p$ } \\
\hline\hline
$1$ &  $3 \mid x$ and $20 \nmid x$ &
\begin{tabular}{@{}c@{}}$x=3^{p-2}\cdot w_1^p$ \\ $x^2+20r^2=w_2^p$ \end{tabular} & 
$w_2^{p} - 3^{2p-4} \cdot w_1^{2p}=20r^2$  
& 46914 \\

\hline

$2$ & $3 \mid x$ and $5\mid x$ and $2\nmid x$ &
\begin{tabular}{@{}c@{}}$x=3^{p-2}\cdot 5^{p-1}\cdot w_1^p$ \\ $x^2+20r^2=5w_2^p$ \end{tabular} & 
$w_2^{p}-3^{2p-4}\cdot 5^{2p-3}\cdot w_{1}^{2p}=4r^2$
& 98461 \\

\hline

$3$ & $3 \mid x$ and $4\mid x$ and $5\nmid x$ &
\begin{tabular}{@{}c@{}}$x=2^{p-2}\cdot 3^{p-2}\cdot w_1^p$ \\ $x^2+20r^2=2^{2}\cdot w_{2}^{p}$ \end{tabular} & 
$w_2^{p}-2^{2p-6}\cdot 3^{2p-4}\cdot w_{1}^{2p}=5r^2$
& 91314 \\

\hline

$4$ & $3 \mid x$ and $4\mid x$ and $5\mid x$ &
\begin{tabular}{@{}c@{}}$x=2^{p-2}\cdot 3^{p-2}\cdot 5^{p-1}\cdot w_1^p$ \\ $x^2+20r^2=2^{2}\cdot 5\cdot w_{2}^{p}$ \end{tabular} & 
$w_2^{p}-2^{2p-6}\cdot 3^{2p-4}\cdot 5^{2p-3}\cdot w_{1}^{2p}=r^2$
& 142861 \\

\hline

$5$ & $3 \nmid x$ and $2 \nmid x$ and $5 \mid x$  & \begin{tabular}{@{}c@{}}$x=5^{p-1} w_1^p$ \\ $x^2+20r^2= 3^{p-2}\cdot5w_2^p$ \end{tabular} & $3^{p-2}\cdot w_2^p-5^{2p-3}w_1^{2p}=4r^2$ 
&  34286 \\

\hline

$6$ & $3 \nmid x$ and $4 \mid x$ and $5 \mid x$  & \begin{tabular}{@{}c@{}}$x=2^{p-2}5^{p-1} w_1^p$ \\ $x^2+20r^2=2^2 3^{p-2} 5w_2^p$ \end{tabular} & $ 3^{p-2}w_2^p-2^{2p-6}5^{2p-3}w_1^{2p}=r^2$ 
&  $78880$ \\

\hline

$7$ & $3 \nmid x$ and $4 \mid x$ and $5 \nmid x$  & \begin{tabular}{@{}c@{}}$x=2^{p-2}w_1^p$ \\ $x^2+20r^2=2^2 3^{p-2}w_2^p$ \end{tabular} & $ 3^{p-2}w_2^p-2^{2p-6}w_1^{2p}=5r^2$ 
&  $27047$ \\

\hline

$8$ & $3 \nmid x$ and $2 \nmid x$ and $5 \nmid x$  & \begin{tabular}{@{}c@{}}$x= w_1^p$ \\ $x^2+20r^2=3^{p-2}w_2^p$ \end{tabular} & $3^{p-2}w_2^p-w_1^{2p}=20r^2$ 
&  $23457$ \\

\hline

$9$ & $3 \mid x$ and $2 \mid x$, $4, 5 \nmid x$ & \begin{tabular}{@{}c@{}}$x=2^{p-1}\cdot 3^{p-2} \cdot w_1^p$ \\ $x^2+20r^2= 2w_{2}^{p}$ \end{tabular}& $w_2^p - 2^{2p-3}3^{2p-4}w_{1}^{2p}=10r^{2}$ & 2\\

\hline

$10$ & $3 \mid x$ and $2, 5 \mid x$, $4 \nmid x$ & \begin{tabular}{@{}c@{}}$x=2^{p-1}\cdot 3^{p-2} \cdot 5^{p-1} \cdot w_1^p$ \\ $x^2+20r^2= 2 \cdot 5 \cdot w_{2}^{p}$ \end{tabular}& $w_2^p - 2^{2p-3}3^{2p-4}5^{2p-3}w_{1}^{2p}=2r^{2}$ &  2\\

\hline

$11$ & $3 \nmid x$ and $2, 5 \mid x$ and $4 \nmid x$  & \begin{tabular}{@{}c@{}}$x=2^{p-1}5^{p-1} w_1^p$ \\ $x^2+20r^2=2\cdot5 \cdot 3^{p-2} w_2^p$ \end{tabular} & $3^{p-2} w_2^p-2^{2p-3}5^{2p-3}w_1^{2p}=2 r^2$ 
& 2 \\

\hline

$12$ & $3 \nmid x$ and $2 \mid x$ and $4, 5 \nmid x$  & \begin{tabular}{@{}c@{}}$x=2^{p-1} w_1^p$ \\ $x^2+20r^2=2 \cdot 3^{p-2} w_2^p$ \end{tabular} & $3^{p-2} w_2^p-2^{2p-3}w_1^{2p}=2 r^2$ 
&  2 \\
\hline

\end{tabular}
\end{adjustbox}
\caption{Descent cases.}
\label{tab:descent}
\end{table}

Finally, we address descent cases 9--12 in Table~\ref{tab:descent}. 
For example, let us consider descent Case 9.
 If $p\geq 3$, then 4 divides $x$.
This contradicts our assumption that 4 does not divide $x$.
Thus $p=2$.
Similar arguments show that $p$ is bounded by $2$ for descent cases 10--12.

\section{An incredibly efficient sieve} 
 \label{sec:bhvbound}

For descent cases 1--4 in Table~\ref{tab:descent}, we make drastic computational improvements by very quickly discarding infeasible values of the prime exponent $p$
that cannot produce solutions to equation~\ref{eq:mainthm}. We primarily achieve this via an application of prior work of the fourth author 
\cite[Theorem 1]{VP21}, which is based on the Primitive Divisor Theorem due to Bilu, Hanrot and Voutier \cite{BHV}.

\begin{theorem}[Patel] \label{thm:Patel} Let $C_{1} \geq 1$ be a square free integer and $C_{2}$ a positive integer. Assume that $C_{1}C_{2} \not\equiv 7 \pmod{8}$. Let $p$ be a prime for which
 \begin{equation*}
     C_{1}x^{2}+C_{2}=y^{p}, \qquad \gcd(C_{1}x^{2}, C_{2}, y^{p})=1 
 \end{equation*}
has a solution $(x, y) \in \Z_{>0}$. Write $C_{1}C_{2} = cd^{2}$, where $c$ is squarefree.
Then one of the following holds:
\begin{itemize}
    \item $p \leq 5$;
    \item $p=7$ and $y=3, 5$ or $9$;
    \item $p$ divides the class number of $\Q(\sqrt{-c})$;
    \item $p$ divides $\left(q-\left(\frac{-c}{q}\right)\right)$, where $q$ is a prime $q \mid d$ and $q \nmid 2c$. 
\end{itemize}
\end{theorem}

 Application of Theorem~\ref{thm:Patel} dramatically reduces the number of equations that need resolving in cases 1--4. This is evident from column 4 of Table~\ref{tab:countingeqs}. Unfortunately, we are unable to apply Theorem~\ref{thm:Patel} to cases 5--8 to obtain similar significant computational savings. We highlight that this is incredibly helpful in particular for case 4, where the bound obtained through appealing to Theorem~\ref{thm:Mignotte} is particularly large ($p \leq 142,861$) and would not be tractable under computations  performed in previous works \cite{Garcia2019, GarciaPatel2019, GarciaPatel2020}.

\begin{table}[htbp!]
\begin{tabular}{|c|c|c|c|}
\hline
Case & \makecell{Upper bound for $p$\\
        from\\
        Theorem~\ref{thm:Mignotte}}
        & \makecell{Approximate number of 
\\eqns to solve in $w_{1}$ and $w_{2}$\\
  after application of Theorem~\ref{thm:Mignotte}}
&
\makecell{Number of eqns
\\ to solve in $w_{1}$ and $w_{2}$ 
\\ after application of 
\\Theorems~\ref{thm:Mignotte} and ~\ref{thm:Patel}}\\
\hline
1 & 46914 & $3.2\times 10^9$ & 1,404,080\\
\hline
2 & 98461 & $5\times 10^9$ & 1,277,832\\
\hline
3 & 91314 & $2.9\times 10^9$ & 725,588\\
\hline
4 & 142861 & $3.5\times 10^9$ & 661,224\\
\hline
5 & 34286 & $1.9\times 10^9$ & $1.9\times 10^9$ \\
\hline
6 & 78880 & $2.1\times 10^9$ &  $2.1\times 10^9$\\
\hline
7 & 27047 & $9.9\times 10^8$ & $9.9\times 10^8$ \\
\hline
8 & 23457 & $1.7\times 10^9$ & $1.7\times 10^9$ \\
\hline\hline
Total & -- & $2.1\times 10^{10}$ & $6.7\times 10^9$\\
\hline
\end{tabular}
    \caption{Counting the number of equations that need to be resolved.}
    \label{tab:countingeqs}
\end{table}

\section{Small prime exponents: revisited} 
\label{sec:prime57}

In this section, we deal with the prime exponents $5$ and $7$.
This enables us to further reduce the computational effort needed in resolving certain descent equations. 
We begin this section by giving a summary of additional methods required. 

\subsection*{The method of Chabauty} 
Let $X/\Q$ be a curve with genus $g$. Suppose $J_{X}$ has rank $r$, where $J_{X}$ is the Jacobian of $X$. Let $p\geq 3$ be a prime of good reduction for $X$.
If the Chabauty condition $r < g$ is satisfied then the method of Chabauty guarantees the existence of a non-empty set of so-called \emph{annihilating differentials}. 
By studying the zeros of these annihilating differentials, one can find a set of $p$-adic points on $X$ that contain $X(\Q)$.
We refer the reader to \cite[Section 1.6]{CCIKPUSurvey} for a brief overview, or to \cite{MullerChabauty} and \cite{Siksek2015Chabauty} for a comprehensive overview of the method of Chabauty. We note if $X$ is a genus 2 hyperelliptic curve then there is a readily available \texttt{MAGMA} \cite{magma} implementation that computes $X(\Q)$ using Chabauty, provided that the Chabauty condition is satisfied.

\subsection*{Thue equations}
For a fixed prime $p$, each descent case yields ternary equations of the form
\begin{equation}
    \label{eq: Fermatpp2}
    aw_{2}^{p}-bw_{1}^{2p}=cr^2
\end{equation}

In these cases, we let $\tau = w_{2}$ and $\sigma=w_{1}^{2}$. 
This transforms ~\eqref{eq: Fermatpp2} to
\begin{equation}
    \label{eq: Thue}
     a\tau^{p}-b\sigma^{p}=cr^{2}.
\end{equation}

For a fixed value of $r$,  equation ~\eqref{eq: Thue} is a Thue equation of degree $p$. For small enough values of $p$, 
we can solve \eqref{eq: Thue} using the Thue solver in \texttt{MAGMA} \cite{magma} which is based on an algorithm of Bilu and Hanrot \cite{BILU1996373}, and Tzanakis and de Weger \cite{TW89}. 

\subsection{The non-existence of solutions for exponent $5$}
Setting $p=5$ and making appropriate substitutions
in \eqref{eq:mainthm} yields a genus 2 hyperelliptic curve $C$. 
By the infamous theorem of Faltings \cite[Theorem 1]{Faltings}, $C$ has finitely many rational points, and hence, integral points.
With the exception of one descent case, we are able to determine $C(\Q)$ using the method of Chabauty. 
Thus we are able to find a resolution to equation~\ref{eq:mainthm} with exponent $p=5$. Key data is recorded in 
Table \ref{tab:Chabauty}.

\medskip

    \begin{table}[htbp!]
\begin{adjustbox}{width =\textwidth}
\begin{tabular}{|c|c|c|c|c|}
\hline
Descent case & Change of variables & $C$ & Rank Bound of $J_C$ & $C(\Q)$\\
\hline
1 & $Y=10r/w_1^5$, $X = w_2/w_1^2$ &$Y^2=5X^5-5\cdot 3^6$ & 1 & $\{(9, \pm 540), \infty\}$\\
\hline
2 & $Y=2r/w_1^5$, $X = w_2/w_1^2$ &$Y^2=X^5-3^{6}\cdot 5^7$ & 0 & $\{\infty\}$\\
\hline 
3 & $Y=5r/w_1^5$, $X = w_2/w_1^2$ &$Y^2=5X^5-5\cdot 2^4\cdot 3^{6}$ & 2 & unverified \\
\hline 
4 & $Y=r/w_1^5$, $X = w_2/w_1^2$  &$Y^2=X^5-2^{4}\cdot 3^{6}\cdot 5^7$ & 0 & $\{\infty\}$\\
\hline
5 & $Y=2r/w_1^5$, $X = w_2/w_1^2$ &$Y^2=3^3\cdot X^{5}-5^{7}\cdot 5$ & 0 & $\{\infty\}$\\
\hline
6 & $Y=r/w_1^5$, $X = w_2/w_1^2$  & $Y^2=3^3\cdot X^5-2^4\cdot 5^7$& 0 & $\{\infty\}$\\
\hline
7 & $Y=5r/w_1^5$, $X = w_2/w_1^2$ & $Y^2=3^{3}\cdot5 X^5- 2^4 \cdot 5$ & 1& $\{\infty\}$\\
\hline
8 & $Y=10r/w_1^5$, $X = w_2/w_1^2$ & $Y^2=3^3\cdot 5X^5 - 5$ & 0 & $\{\infty\}$\\
\hline

\end{tabular}
\end{adjustbox}
\caption{Treating $p=5$ using the method of Chabauty}
\label{tab:Chabauty}

\end{table}

For the convenience of the reader, we expand upon the details in descent cases 1 and 7. Firstly, descent case 1 gives rise to 
the genus 2 hyperelliptic curve $C$ as stated in Table~\ref{tab:Chabauty}.
Using the Chabauty implementation in \texttt{MAGMA} \cite{magma}, we determine all rational points on $C$; these have been recorded in Table~\ref{tab:Chabauty}.
The point at infinity 
$\infty\in C(\Q)$ 
corresponds
to $x=0$, whilst the points $(9,\pm 540)\in C(\Q)$ imply that $3\mid r$, contradicting the coprimality of $x$ and $r$.

\medskip

We now consider descent case 7. 
We deduce that  
the genus 2 hyperelliptic curve obtained, $C$ (stated in Table~\ref{tab:Chabauty}),
has Jacobian with a rank bounded above by 1.
With the help of \texttt{MAGMA} \cite{magma}, we found a point $D\in J(\Q)$ that has infinite order, where 
\[
D=(x^2 - x + 2/3, 15x -10),
\]
given in Mumford representation, 
thus the rank of $J(\Q)$ is 1. 
Hence we are able to apply the method of Chabauty, and using the \texttt{MAGMA} \cite{magma} implementation, we find that $C(\Q)=\{\infty\}$.

\medskip 

Unfortunately, we find that 
descent case 3 is intractable under the method of Chabauty
as we are unable to ascertain the rank of the Jacobian of the genus 2 hyperelliptic curve $C$ (stated in Table~\ref{tab:Chabauty}),
hence 
we cannot verify that the Chabauty condition is satisfied. 
Instead, 
we let $\sigma=w_{2}$ and $\tau=w_{1}^{2}$  to yield the Thue equation
\begin{equation}
    \label{eq: 3Thue5}
        \sigma^{5}-2^{4}\cdot 3^{6}\cdot\tau^{5}=5r^{2}
\end{equation}
for a fixed value of $r$. We use \texttt{MAGMA}'s \cite{magma} Thue solver, and  find that the only solutions to \eqref{eq: 3Thue5}, for $1\leq r\leq 10^6$, are 
$(\sigma, \tau, r) \in \{(5, 0, 5^2), (5^3, 0, 5^{7}), (5\cdot 7^2, 0, 5^{2}\cdot 7^{5})\}$.
Our computations with prime exponent 5 thus prove the following proposition.

\begin{proposition}\label{prop:prime5}
    The equation
\begin{equation*}  
(x-4r)^3 + (x-3r)^3 + (x-2r)^3+(x-r)^3 + x^3 + (x+r)^3+(x+2r)^3 + (x+3r)^3 + (x+4r)^3 = y^5 
\end{equation*}
with $x, r, y \in \Z$, $\gcd(x, r) = 1$ and $0< r <10^6$ has no solutions with $xy \neq 0$.
\end{proposition}

\subsection{The non-existence of solutions for exponent $7$}
We consider descent cases 1--4. 
An appropriate substitution (see Table \ref{tab:BHVequations}),  transforms the descent equation to an equation of the form
\begin{equation}
    \label{eq:BHVeq}
    C_{1}X^2+C_{2}=w_{2}^{p},
\end{equation}
where $C_{1}, C_{2}\in \Z_{>0}$, $C_{1}$ is squarefree and $C_{1}C_{2}\not\equiv 7$ (mod 8).
Suppose $p=7$.
By Theorem~\ref{thm:Patel}, 
any solution to \eqref{eq:BHVeq} requires $w_{2}\in\{3, 5, 9\}$. 
One can easily verify that for any such $w_{2}$, the corresponding solution $(x, y, r)$ to \eqref{eq:mainthm} would satisfy $\gcd(x, r)>1$.
This contradicts our hypotheses.

    \begin{table}[htbp!]
\begin{adjustbox}{width =\textwidth}
\begin{tabular}{|c|c|c|c|c|}
\hline
\text{Case} & Conditions on $x$ &  \makecell{Descent equation} 
& \makecell{Change of variables} & \makecell{Descent equation \\ after transformation} \\
\hline
1 & $3 \mid x$ and  $20 \nmid x$ & $x^2 + 20r^2 = w_{2}^{p}$ &  $x=X$ & $X^2 + 20r^2 = w_{2}^{p}$ \\
\hline
2 & $3 \mid x$ and  $5\mid x $ and $ 2 \nmid x$ & $x^2 + 20r^2 = 5w_{2}^{p}$
 & $x = 5X$ &   $5X^2 + 4r^2 = w_{2}^{p}$ \\
\hline
3 & $3 \mid x$ and $4\mid x$ and $5 \nmid x$ & $x^2 + 20r^2 = 2^{2}w_{2}^{p}$ & $x = 2X$ &
        $X^2 + 5r^2 = w_{2}^{p}$ \\
\hline
4 & $3 \mid x$ and $4\mid x$ and $5\mid x$ & $x^2 + 20r^2 = 2^{2}\cdot 5\cdot w_{2}^{p}$ &
$x = 10X$ &
        $5X^2 + r^2 = w_{2}^{p}$ \\
\hline
\end{tabular}
\end{adjustbox}
\caption{Change of variables in descent equations.}
\label{tab:BHVequations}
\end{table}

Thus, we have proven the following.
\begin{proposition}\label{prop:prime7}
    The equation
\begin{equation*} 
(x-4r)^3 + (x-3r)^3 + (x-2r)^3+(x-r)^3 + x^3 + (x+r)^3+(x+2r)^3 + (x+3r)^3 + (x+4r)^3 = y^7 
\end{equation*}
with $x, r, y \in \Z$, $\gcd(x, r) = 1$, $r >0$
and with 
$3 \mid x$  has no solutions with $xy \neq 0$.
\end{proposition}

\section{Equation elimination: 6.7 billion to solve }
 \label{sec: milliontosolve}

To solve all (approximately $6.7\times 10^9$) of
the remaining equations in variables $w_{1}$ and $w_{2}$,
we  employ a combination of different criteria  to  finally resolve each one of the remaining eight descent cases. The implementation of these tests in \texttt{MAGMA} \cite{magma} eliminates equations with no solutions.

\subsection{Sophie Germain's empty set criteria}
\label{subsec:elim1}

First, we  apply the ``empty set'' criteria given in the following lemma. This result is based on work of Sophie Germain and gives a criteria for the nonexistence of solutions to \eqref{eq: Fermatpp2}. For each prime $p$, the criteria constructs an auxiliary prime $q$ and a set $\mathcal{S}(p,q)$. Since the elements of $\mathcal{S}(p,q)$ lie in the finite field $\F_{q}$, the criteria is computationally efficient. Upon comparing Table~\ref{tab:countingeqs} with Tables~\ref{tab:case1234} and \ref{tab:case5678} in Section~\ref{sec:SolveIT}, we see that this criterion is indeed powerful as only a small proportion of equations survive after its application.

\begin{lemma} \cite[Lemma 7.2.1]{Patel17}
    \label{lem:Sophiecriterion}
Let $p \ge 3$ be a prime. Let $a$, $b$ and $c$ be positive coprime integers.
Let $q=2kp+1$ be a prime such that $q\nmid a$.
Define
\begin{equation*}
\mathcal{S}^{\prime}(p,q)=\{ \eta^{2p} \; : \; \eta \in \F_q \}
=\{0\} \cup \{ \zeta \in \F_q^* \; : \; \zeta^{k}=1\}
\end{equation*}
and
$$
\mathcal{S}(p,q)=\left\{ \zeta \in \mathcal{S}^{\prime}(p,q) \; : \; ((b \zeta+c)/a)^{2k} \in \{0,1\} \right\} \, .
$$
If $\mathcal{S}(p,q)=\emptyset$,
then equation~\eqref{eq: Fermatpp2} 
does not have integral solutions.
\end{lemma}

\subsection{Local solubility}\label{subsec:localsol}
To the equations that survive Lemma~\ref{lem:Sophiecriterion}, we apply a classical local solubility test to eliminate further equations. We outline the procedure below.

For a fixed valued of $r$, the ternary equations from the descent are of the form  
\begin{equation} \label{eq:descent} aw_{2}^{p}-bw_{1}^{2p}=c \end{equation}
and satisfy $\gcd(a, b, c) =1$. If $g = \text{Rad}(\gcd(a, c))>1$, the $\gcd$ assumption yields $g\mid w_{1}$ and we can write $w_{1}=gw_{1}'$, $a=gd$ and $c=gf$. Thus from 
equation~\eqref{eq:descent} we obtain
$$dw_{2}'^{p}-ew_{1}'^{2p}=f,$$
where $w_{2}'=w_{2}$ and $e=g^{2p-1}$. 
The repetition of similar arguments yields 
\begin{equation} \label{eq:descentcoprime} 
D\lambda^{p}-E\mu^{2p}=F\nu,
\end{equation} 
where $D, E$ and $F$ are now pairwise coprime. Now we can study the local behaviour of \eqref{eq:descentcoprime} at different primes. 
Let $q$ be a prime.
\begin{itemize}
    \item Suppose $q\mid D$. 
    From reducing \eqref{eq:descentcoprime} modulo $q$, it follows that $-EF$ is a quadratic residue modulo $q$.

    \item Suppose $q \equiv 1 \pmod{p}$. Then we can view \eqref{eq:descentcoprime} as an equation in $\F_{q}$. If we assume that $q\mid D$, we see that \eqref{eq:descentcoprime} can only have (non-trivial) solutions if $-F/E$ is a $2p$-power in $\F_{q}$. Similarly, if $q\mid E$ (resp. $q\mid F$), then \eqref{eq:descentcoprime} has solutions over $\F_{p}$ if $F/D$ (resp. $E/D$) is a $p$-power.  
    
    \item Suppose $q\in \{2, 3, 5, 7, p\}$ or $q\mid DEF$. Then we can use the local solubility implementation in \texttt{MAGMA} \cite{magma} to test if the rational projective curve given by \eqref{eq:descentcoprime} has $\Q_{q}$--rational points. 
\end{itemize}

\subsection{Descent over number fields}
To the equations that survive Sections \ref{subsec:elim1} and \ref{subsec:localsol}, we apply a further descent over number fields. 
With $D,\; E,\; F$ as in \eqref{eq:descentcoprime}, we write  
$$ E'= \prod_{\ord_{q} E \text{ is odd }} q.$$ 
Hence, $EE'=s^{2}$ for some $s\in\Z$. 
Write $DE'=r$ and $FE'=n^2m$ with $m$ squarefree. Then we can rewrite \eqref{eq:descentcoprime} as
$$ r\rho^{p} = (s\kappa^{p}+n\sqrt{-m})(s\kappa^{p}-n\sqrt{-m}),$$ 
where $\rho=\lambda^{p}$ and $\kappa=\mu$. Let $K= \Q(\sqrt{-m})$ and $\mathcal{O}$ be its ring of integers. 
Let 
\[
\mathfrak{G}=\{\mathcal{P}\subset\mathcal{O} : \mathcal{P}\text{ is a prime ideal and } \mathcal{P}\mid r\text{ or }\mathcal{P}\mid 2n\sqrt{{-m}}\}.
\]
Let $\tau = (s\kappa^{p}+n\sqrt{-m})$.  
If $\mathcal{P} \not\in \mathfrak{G}$, then $\ord_{\mathcal{P}}(\tau \overline{\tau})=p\ord_{\mathcal{P}}(\rho)$. By assumption, $\ord_{\mathcal{P}}(\tau+ \overline{\tau})=\ord_{\mathcal{P}}(2n\sqrt{-m})=0$. Hence, the equivalence class of $\tau$ in $K^{\ast}/(K^{\ast})^{p}$ is an element of the ``$p$-Selmer group''
$$ K(\mathfrak{G}, p) = \{ \varepsilon\in K^{\ast}/(K^{\ast})^{p}: \ord_{\mathcal{P}}(\varepsilon) \equiv 0 \pmod{p} \text{ for all } \mathcal{P}\not\in\mathfrak{G}\}.$$
This is an $\F_{p}$ vector space that can be computed by \texttt{MAGMA} \cite{magma} using the command  \texttt{pSelmerGroup}. 
Then, 
\begin{equation}
    \label{eq:pselmer}
    (s\kappa^{p}+n\sqrt{-m})=\varepsilon\eta^{p},
\end{equation}
for some $\eta \in K^{\ast}$, and $\varepsilon \in \mathcal{E} := \{\varepsilon \in K(\mathfrak{G}, p) : \text{norm}(\varepsilon)/r\in (\Q^{\ast})^{p}\}.$ 
This yields the following two criteria analogous to Lemma~\ref{lem:Sophiecriterion}, and local solubility techniques over $\Q$.

\begin{lemma} \cite[Lemma 9.2.1]{Patel17} 
Let $K=\Q(\sqrt{-m})$, and let $\mathfrak{q}$ be a prime ideal of $K$. Suppose one of the following holds:
\begin{itemize}
    \item[(i)]  $\ord_{\mathfrak{q}}(s),\;  \ord_{\mathfrak{q}}(n\sqrt{-m}),\; \ord_{\mathfrak{q}}(\varepsilon)$ are pairwise distinct modulo $p$;
    \item[(ii)] $\ord_{\mathfrak{q}}(2s),\; \ord_{\mathfrak{q}}(\varepsilon),\;  \ord_{\mathfrak{q}}(\overline{\varepsilon})$ are pairwise distinct modulo $p$;
    \item[(iii)]  $\ord_{\mathfrak{q}}(2n\sqrt{-m}),\; ord_{\mathfrak{q}}(\varepsilon),\; \ord_{\mathfrak{q}}(\overline{\varepsilon})$ are pairwise distinct modulo $p$.
\end{itemize}
Then there is no $\kappa \in \Z$ and $\eta \in K$ satisfying ~\eqref{eq:pselmer}.
\end{lemma}

\begin{lemma} \cite[Lemma 9.2.2]{Patel17}  
Let $q = 2kp + 1$ be a prime. Suppose $q{\mathcal{O}} = \q_1 \q_2$ where $\q_1, \q_2$ are distinct,
and such that $\ord_{\q_{j}} (\varepsilon) = 0$ for $j = 1, 2$. Let
$$\chi^{\prime}(p, q) = \{\eta^{p} : \eta \in \F_{q} \}.$$
Let
$$ \chi(p, q) = \{\zeta \in \chi^{\prime}(p, q) : ((s\zeta + n\sqrt{-m})/\varepsilon)^{2k} \equiv 0 \text{ or } 1 \text{ mod } \q_j \text{ for } j = 1, 2\}.$$
Suppose $\chi(p, q) = \emptyset$. Then there is no $\kappa \in \Z$ and $\eta \in K$ satisfying ~\eqref{eq:pselmer}.
\end{lemma}

\section{The final resolution!}\label{sec:SolveIT}

In this section, we apply the techniques described in Sections~\ref{sec:descandbound} and \ref{sec: milliontosolve} in order to resolve all the remaining equations.  
To this end, we run a \texttt{MAGMA} \cite{magma} script
which implements the mathematical bounds and tests outlined in Sections~\ref{sec:descandbound} and \ref{sec: milliontosolve}. 
This completes the proof of Theorem~\ref{thm:main}.

\subsection{Full resolution of  Cases 1 to 4}
\label{sec:case1234}

We first record computational data for Cases 1--4. These cases 
are amenable to Theorem~\ref{thm:Patel}, which provides an extremely fast test to eliminate exponents. Further, we have already dealt with small prime exponents $5$ and $7$ in Section~\ref{sec:prime57}.
Recall our exponent is bounded above via Mignotte's Theorem (see Table~\ref{tab:descent}) and thus we have a finite number of equations to resolve in two unknown variables. We take all remaining equations through the tests outlined in Section~\ref{sec: milliontosolve}, and record the outcome in Table~\ref{tab:case1234}.

\begin{table}[htbp!]
\begin{adjustbox}{width =\textwidth}
\begin{tabular}{|c|c|c|c|c|}
\hline
\text{Exponent }$p$ & \makecell{Number of eqns\\ surviving
\\ empty set criterion} 
& \makecell{Number of eqns \\surviving local \\solubility tests} & \makecell{Number of eqns \\surviving \\further descent} & \makecell{Thue eqns \\ not solved  \\ by \texttt{MAGMA}} \\
\hline
11 & 6143 & 703 & 0 & 0\\
\hline
13 & 422 & 127 & 0 & 0\\
\hline
17 & 208 & 37 & 0 & 0\\
\hline
19 & 794 & 341 & 0 & 0\\
\hline
23 & 5 & 1 & 0 & 0\\
\hline
29 & 0 & 0 & 0 & 0\\
\hline
31 & 5 & 3 & 0 & 0\\
\hline
37 & 1 & 0 & 0 & 0\\
\hline
$41\leq p \leq 142861$& 0 & 0 & 0 & 0\\
\hline
\end{tabular}
\end{adjustbox}
\caption{A breakdown of the computations for Cases 1--4} 
\label{tab:case1234}

\end{table}

We remark once more that application of Theorem~\ref{thm:Patel} and  Chabauty techniques significantly reduced the number of equations that needed to be resolved, thereby significantly reducing the computation time (see Table~\ref{tab:comptime14}).

\medskip

Moreover, Theorem~\ref{thm:Patel} allowed us to quickly bypass certain exponents that would otherwise pass the insolubility tests of Section~\ref{sec: milliontosolve} and give rise to intractable Thue equations.
One such example arises when considering descent case 3. We choose the  pair $(r,p)=(390625, 17)$.
The ternary descent equation arising from this case is given by
\begin{equation}
    \label{eq:ThueFailure}
        w_{2}^{17}-2^{28}\cdot 3^{30}\cdot w_{1}^{34}=5\cdot 390625^2.
\end{equation}
We found that the techniques outlined in Section~\ref{subsec:elim1} fail to resolve \eqref{eq:ThueFailure}. 
This is merely one example where the strategies outlined in previous work \cite{Garcia2019, GarciaPatel2019, GarciaPatel2020},
are alone insufficient to prove  Theorem~\ref{thm:main}. Another example of an intractable Thue equation bypassing all tests of Section~\ref{sec: milliontosolve}, arising from descent case $2$ with exponent $19$ and $r=262144$ can be found in \cite{CCIKPUSurvey}.

\subsection{Full resolution of  Cases 5 to 8}
\label{sec:case5678}

We now turn to Cases 5--8. Unfortunately,
we are unable to apply Theorem~\ref{thm:Patel} to quickly discard exponents, or deal with the exponent $7$ in an efficient manner. Computationally, a full resolution is still achieved, albeit, much less efficiently. See Table~\ref{tab:comptime14} for a comparison of the computational times. 
We recall that we have dealt with the exponent $p=5$ in Section~\ref{sec:prime57} using Chabauty.

In contrast with cases 1--4  
we now find equations that survive Sophie Germain's criteria,  local solubility tests and the second descent (over a number field). These remaining equations all have exponent $7$ and occur at the following values of $r$:
\[
 r\in \{2401,277360, 352832, 389176, 729296, 809336, 826864, 903464, 979616\}.
\]
These 9 equations are resolved using the Thue solver in \texttt{Magma} which is based on an algorithm of Bilu and Hanrot \cite{BILU1996373}, and Tzanakis and Weger \cite{TW89}. No solutions are found.

\begin{table}[H]
\begin{adjustbox}{width =\textwidth}
\begin{tabular}{|c|c|c|c|c|}
\hline
\text{Exponent }$p$ & \makecell{Number of eqns\\ surviving
\\ empty set criterion} 
& \makecell{Number of eqns \\surviving local \\solubility tests} & \makecell{Number of eqns \\surviving \\further descent} & \makecell{Thue eqns \\ not solved  \\ by \texttt{MAGMA}} \\
\hline
7 & 117938 & 57619 & 9 & 0\\
\hline
11 & 21525 & 12246 & 0 & 0 \\
\hline
13 & 4447 & 2455 & 0 & 0 \\
\hline
17 & 950 & 473 & 0 & 0 \\
\hline
19 & 2374 & 1472 & 0 & 0 \\
\hline
23 & 36 & 15 & 0 & 0 \\
\hline
29 & 18 & 11 & 0 & 0 \\
\hline
31 & 64 & 39 & 0 & 0 \\
\hline
37 & 3 & 2 & 0 & 0\\
\hline
41 & 1 & 0 & 0 & 0\\
\hline
$43\leq p \leq 78880$& 0 & 0 & 0 & 0 \\
\hline
\end{tabular}
\end{adjustbox}
\caption{A breakdown of the computations for Cases 5--8} 
\label{tab:case5678}
\end{table}

\subsection{Computational Data}
\label{sec:compdat}
We list the approximate computational times for each descent case here.  Computations ran on an \texttt{Intel(R) Xeon(R) CPU E5-2683 v3 @ 2.00GHz}, and split over 8 processors, took roughly 19 days to complete.

    \begin{table}[H]
\begin{tabular}{|c|c|c|c|}
\hline
Descent case & Computation Time & Descent case & Computation Time\\
\hline
1 & $\sim 21$ minutes & 5 & $\sim 160$ hours\\
\hline
2 & $\sim 23$ minutes & 6 & $\sim 450$ hours\\
\hline
3 & $\sim 10$ minutes & 7 & $\sim 50$ hours\\
\hline
4 & $\sim 16$ minutes & 8 & $\sim 61$ hours\\
\hline
\end{tabular}
\caption{Computation times for descent cases 1--8}
\label{tab:comptime14}

\end{table}

\bibliographystyle{siam}
\bibliography{ref.bib}

\end{document}